\documentclass[12pt]{amsart}
\usepackage{amssymb}
\usepackage{amsmath, amscd}
\usepackage{amsthm}
\usepackage[table,xcdraw]{xcolor}
\usepackage{ulem}
\usepackage{tikz}
\usepackage{circuitikz}
\usepackage[colorlinks=true, linkcolor=blue,urlcolor=blue]{hyperref}

\newtheorem{theorem}{Theorem}[section]

\newtheorem{proposition}[theorem]{Proposition}
\newtheorem{lemma}[theorem]{Lemma}

\newtheorem{corollary}[theorem]{Corollary}
\theoremstyle{definition}

\newtheorem{example}[theorem]{Example}

\begin{document}
	
\author[P. Danchev]{Peter Danchev}
\address{Institute of Mathematics and Informatics, Bulgarian Academy of Sciences, 1113 Sofia, Bulgaria}
\email{danchev@math.bas.bg; pvdanchev@yahoo.com}

\author[M. Doostalizadeh]{Mina Doostalizadeh}
\address{Department of Mathematics, Tarbiat Modares University, 14115-111 Tehran Jalal AleAhmad Nasr, Iran}
\email{d\_mina@modares.ac.ir; m.doostalizadeh@gmail.com}

\author[M. Esfandiar]{Mehrdad Esfandiar}
\address{Department of Mathematics and Computer Science Shahed University Tehran, Iran}
\email{mehrdad.esfandiar@shahed.ac.ir}

\author[O. Hasanzadeh]{Omid Hasanzadeh}
\address{Department of Mathematics, Tarbiat Modares University, 14115-111 Tehran Jalal AleAhmad Nasr, Iran}
\email{o.hasanzade@modares.ac.ir; hasanzadeomiid@gmail.com}

\title[Rings such that $u-1$ lies in $J^{\#}(R)$ for each unit $u$]{{\small Rings such that} \large$u-1$ {\small lies in $J^{\#}(R)$ \\ for each unit} \large$u$}
\keywords{$UJ^{\#}$ rings, semi-potent rings, regular rings, Boolean rings}
\subjclass[2010]{16S34, 16U60, 20C07}

\maketitle




\begin{abstract}
We investigate the so-called {\it $UJ^{\#}$ rings}, a new type of rings in which every unit can be written as $1+j$ with $j\in J^{\#}(R)$. These rings were defined and studied by Saini-Udar in Czechoslovak Math. J. (2025) under the name {\it $\sqrt{J}U$ rings}. (See \cite{SU}.)

This class extends both the classes of UU and UJ rings, but also has its own special properties. In this study, we present some additional results about $UJ^{\#}$ rings that supply those from \cite{SU} explaining their connections with Dedekind-finite, semi-potent and Boolean rings, respectively, as well as we give several characterizations in this direction. We also examine how these rings behave under common ring constructions and find conditions for group rings to be $UJ^{\#}$. Moreover, our establishments shed a clearer picture of how unit elements interact with radical-like parts of a ring.
\end{abstract}

\section{Introduction and Motivation}

Throughout the current paper, all rings under consideration are assumed to be associative with identity. For such a ring $R$, the Jacobson radical, the group of units, the set of nilpotent elements, the set of central elements and the set of idempotent elements are denoted by $J(R)$, $U(R)$, $Nil(R)$, $Z(R)$ and $Id(R)$, respectively.

Besides, letting $\alpha : R \to R$ be a ring endomorphism, we denote by $R[x; \alpha]$ the ring of {\it skew polynomials} over $R$ with multiplication defined as $xr = \alpha(r)x$ for all $r \in R$. In particular, $R[x] = R[x; 1_R]$ is the {\it ordinary ring of polynomials} over $R$. Likewise, $R[[x; \alpha]]$ denotes the ring of {\it skew formal power series} over $R$, that is, all formal power series in $x$ with coefficients from $R$ with multiplication defined via $xr = \alpha(r)x$ for all $r \in R$. In particular, $R[[x]] = R[[x; 1_R]]$ is the ring of {\it formal power series} over $R$.

In \cite{wang}, the set \[J^{\#}(R)=\{z\in R : z^n\in J(R),\ \text{for some} n\ge 1\}\] was introduced as a subset of \(R\), which is {\it not} necessarily a subring, properly containing $J(R)$. It is evident that $N(R)$ is too a (proper) subset of $J^{\#}(R)$.

We now provide the non-expert reader with some short historical retrospection of the most central terminology: A ring $R$ is said to be an {\it UU ring} if $U(R) = 1 + Nil(R)$. This concept was first defined in \cite{CUU}, and later on examined in greater depth in \cite{DL}. It was established in the latter article that a ring $R$ is strongly nil-clean if, and only if, it is an exchange UU ring. Furthermore, it was proved in \cite{karimi} that $R$ is an UU ring if, and only if, every unit in $R$ is uniquely nil-clean.

It is well known that $1 + J(R) \subseteq U(R)$. Motivated by this inclusion and the aforementioned concept of UU rings, the notion of a {\it JU ring} as the ring $R$ for which $U(R) = 1 + J(R)$ was introduced in \cite{Drings}. This concept was further explored in \cite{kosan1}, where it was referred to as a {\it UJ ring}. It was shown there that a ring $R$ is a UJ ring if, and only if, every clean element of $R$ is $J$-clean. The structural properties of UJ rings, along with their potential applications in various areas of non-commutative ring theory, have been the subject of several recent investigations \cite{Drings,kosan1,kosannot,leroy}.

Recently, in \cite{SU} the authors defined a ring to be a {\it $\sqrt{J}U$ ring}, provided $U(R) = 1 + \sqrt{J(R)}$, where $\sqrt{J(R)}=J^{\#}(R)$. For convenience of the records, we hereafter will use the second notation.

Since $J(R)\cup Nil(R)\subseteq J^{\#}(R)$, it is pretty clear that both $UU$ and $UJ$ rings satisfy this condition and are, therefore, $UJ^{\#}$ rings; however, the converse does {\it not} hold in general, because we exhibit in the sequel a specific counterexample to support this claim.

Next, we define an element $a \in R$ to be {\it (strongly) $J^{\#}$-clean} if there exists an idempotent element $e \in Id(R)$ such that $a - e \in J^{\#}(R)$ (with the extra commuting condition $ea = ae$ in the strongly $J^{\#}$-clean case).

The primary objective of this research exploration is to provide a comprehensive description of $UJ^{\#}$ rings by comparing their key properties with those of $UU$ and $UJ$ rings, respectively. Additionally, we aim to uncover some novel and distinctive properties of $UJ^{\#}$ rings that are rarely encountered in the existing literature and that non appeared in \cite{SU} as well.

We are now planning to give a brief program of our main material established in what follows: In Section 2, we first study the basic properties of the set $J^{\#}(R)$ for a ring $R$ and use them to construct some properties and examples of specific $UJ^{\#}$ rings. Additionally, we achieve to exhibit some major properties and characterizations of $UJ^{\#}$ rings in various different aspects (see, for instance, Propositions \ref{2.2} and \ref{2.3}, Lemmas \ref{corner}, \ref{equ UQ} and \ref{matrix} plus Theorems \ref{3.5} and \ref{m}). Furthermore, we show that every $UJ^{\#}$ ring is Dedekind-finite (see Lemma \ref{dedkind finite}). In Section 3, we establish some fundamental characterizing properties of $UJ^{\#}$ rings that are mainly stated and proved in Theorems \ref{2.4}, \ref{3.16}, \ref{skew polynomial} and \ref{thm 2 primal}, Corollaries \ref{2.7} and \ref{3.18}, Propositions \ref{3.2} and \ref{pro clean element}. The leitmotif in Section 4 is to find certain necessary and sufficient conditions for the group ring $RG$ to be an $UJ^{\#}$ ring (see Theorems~\ref{UQgroupring} and \ref{finone3}).

\section{Examples and Fundamental Properties of $UJ^{\#}$ Rings}

Let $R$ be an arbitrary ring. For completeness of the exposition, we once again set

\[
J^{\#}(R):=\{z\in R : z^n\in J(R),\ \text{for some } n\ge 1\}.
\]

We next obtain a useful technicality, needed for our successful presentation.

\begin{lemma}\label{1.2}
Let $R$ be any ring. Then, the following assertions hold:

(1) If $a \in J^{\#}(R)$ and $b \in R$ with $ab = ba$, then $ab \in J^{\#}(R)$.
		
(2) For every $a \in R$ and $n \in \mathbb{N}$, we have $a^n \in J^{\#}(R)$ if and only if $a \in J^{\#}(R)$.
		
(3) If $a \in J^{\#}(R)$, then $1 - a \in U(R)$.
		
(4) If $a \in J^{\#}(R) \cap Z(R)$, then $a \in J(R)$.
		
(5) For every ideal $I \subseteq J(R)$, we have $J^{\#}(R/I) = J^{\#}(R)/I$.
		
(6) If $R = \prod R_i$, then $J^{\#}(R) = \prod J^{\#}(R_i)$.

(7) If $a, b\in R$ and $ab\in J^{\#}(R)$, then $ba\in J^{\#}(R)$.

(8) $Nil(R) + J(R) \subseteq J^{\#}(R)$.
\end{lemma}

\begin{proof}
(1) Suppose $a \in J^{\#}(R)$ and $ab = ba$. Thus, for some $n \in \mathbb{N}$, it must be that $a^n \in J(R)$. Hence, $(ab)^n = a^n b^n \in J(R)$, so $ab \in J^{\#}(R)$.
		
(2) Assume $a \in J^{\#}(R)$. Thus, for any $n \in \mathbb{N}$, $aa^{n-1} = a^{n-1}a$, so by (1), $a^n \in J^{\#}(R)$. Conversely, if $a^n \in J^{\#}(R)$, then for some $k \in \mathbb{N}$, $(a^n)^k = a^{nk} \in J(R)$, whence $a \in J^{\#}(R)$.
		
(3) Let $a \in J^{\#}(R)$, so $a^n \in J(R)$ for some $n \in \mathbb{N}$. Thus, $1 - a^n \in U(R)$. But, one sees that
\[
		(1 - a)(1 + a + \cdots + a^{n-1}) = 1 - a^n \in U(R),
\]
implying $1 - a \in U(R)$.
		
(4) Let $a \in J^{\#}(R) \cap Z(R)$. Thus, for all $r \in R$, $ar = ra \in J^{\#}(R)$ viewing (1). So, with (4) in mind, $1 - ar \in U(R)$, and hence $a \in J(R)$.
		
(5) Suppose $a + I \in J^{\#}(R/I)$. Thus, for some $n \in \mathbb{N}$, we derive
\[
		(a + I)^n = a^n + I \in J(R/I) = J(R)/I.
\]
So, there exists $j \in J(R)$ such that $a^n - j \in I \subseteq J(R)$, and therefore $a^n \in J(R)$ giving $a \in J^{\#}(R)$. Consequently, $a + I \in J^{\#}(R)/I$.
		
Conversely, suppose $a + I \in J^{\#}(R)/I$. Thus, there exists $b \in J^{\#}(R)$ with $a - b \in I$ and $b^n \in J(R)$, so that $a = b + j$ for some $j \in J(R)$. Finally, $a^n \in J(R)$ forcing $a + I \in J^{\#}(R/I)$.
		
(6) This follows immediately from the fact that $J(R) = \prod J(R_i)$ and the behavior of powers in product rings.

(7) Suppose $ab\in J^{\#}(R)$. Thus, there exists $n\in \mathbb{N}$ such that $(ab)^n\in J(R)$, so $(ba)^{n+1}=b(ab)^na\in J(R)$.

(8) Let $q + j \in Nil(R) + J(R)$, where $q \in Nil(R)$ with $q^n = 0$, and $j \in J(R)$. Therefore, it is easy to see that $(q + j)^n \in J(R)$, yielding $q + j \in J^{\#}(R)$.
\end{proof}

It is worthy of mentioning that point (6) already appeared in \cite{SU}.

\begin{example}\label{1.3}
For any ring $R$, as we already noticed above, \( J(R) \subseteq J^\#(R) \) and \( Nil(R) \subseteq J^\#(R) \), but the converse is {\it not} necessarily true. E.g., put \(R:=M_2(\mathbb{Z}_2) \) and \( S:=\mathbb{Z}_2[[x]] \). So, one verifies that
\[J^{\#}(R) = \left\{
\begin{pmatrix}
	0 & 0 \\
	0 & 0
\end{pmatrix},
\begin{pmatrix}
	0 & 1 \\
	0 & 0
\end{pmatrix},
\begin{pmatrix}
	0 & 0 \\
	1 & 0
\end{pmatrix}
\right\}
,\]
while we know that $J(R)=(0)$. Furthermore, \(x \in J(S) \subseteq J^\#(S) \), but \( x \notin N(S) \), as suspected.
\end{example}

Although the following statement appeared also in \cite{SU}, we will provide a proof for the reader's convenience.

\begin{proposition}\label{3.8}
For any ring $R$, we always have \( J^\#(R) \cap Id(R) = \{0\} \) and \( J^\#(R) \cap U(R) = \emptyset \).
\end{proposition}

\begin{proof}
If \( e \in J^\#(R) \cap Id(R) \), then Lemma \ref{1.2}(3) works to get that \( 1-e \in Id(R) \cap U(R) \), so \( e=0 \). Now, if \( u \in J^\#(R) \cap U(R) \), then there is \( n \in \mathbb{N} \) such that \( u^n \in J(R) \), and thus \( 0=1-u^{-n}u^n \in U(R) \), which is a contradiction, as expected.
\end{proof}

We next proceed by proving a series of other preliminaries.

\begin{proposition}\label{3.7}
A ring \( R \) is local if and only if \( R=U(R) \cup J^\#(R) \).
\end{proposition}

\begin{proof}
If \( R \) is local, then knowing that \( R=U(R) \cup J(R) \subseteq U(R) \cup J^\#(R) \subseteq R \), we thus deduce \( R=U(R) \cup J^\#(R) \).

Reciprocally, suppose \( R=U(R) \cup J^\#(R) \) and \( a \not\in U(R) \). Then, one has that \( a \in J^\#(R) \), so that Lemma \ref{1.2}(3) applies to detect that \( 1-a \in U(R) \), whence \( R \) is a local ring, as promised.	
\end{proof}

\begin{proposition}\label{3.4}
Let \(R\) be a $UJ^{\#}$ ring. For any unital subring \(S\) of \(R\), if \(S \cap J(R) \subseteq J(S)\), then \(S\) is a $UJ^{\#}$ ring. In particular, the center of \(R\) is a $UJ^{\#}$ ring.
\end{proposition}

\begin{proof}
Given the assumption, it can easily be shown that \(S \cap J^{\#}(R) \subseteq J^{\#}(S)\). Letting $v\in U(S)$, we observe that $\subseteq U(R)$, and since $R$ is $UJ^{\#}$, we infer $v-1\in J^{\#}(R)\cap S\subseteq J^{\#}(S)$. So, $S$ is a $UJ^{\#}$ ring. The rest of the assertion now follows directly from \cite[Ex.5.0]{lam}.
\end{proof}

A {\it rationally closed} subring of $R$ is a subring $S$ such that $U(R) \cap S = U(S)$. Apparently, $R$ is a rationally closed subring of both $R[x]$ and $R[[x]]$. However, $R[x]$ is obviously {\it not} necessarily a rationally closed subring of $R[[x]]$, because $1+x \in U(\mathbb{Z}[[x]]) \cap \mathbb{Z}[x]$, but $1+x \not\in U(\mathbb{Z}[x])$.

\medskip

We, thereby, arrive at the following five routine claims that also appeared in \cite{SU}, but are presented here only for the sake of completeness to be the text more nearly self-contained and friendly to the interested reader.

\begin{corollary}\label{subring}
Let $S$ be a rationally closed subring of $R$, so if $R$ is a $UJ^{\#}$ ring, then $S$ is also a $UJ^{\#}$ ring.
\end{corollary}

\begin{proof}
Since $U(R) \cap S = U(S)$, we also conclude \(S \cap J(R) \subseteq J(S)\). Then, the result follows immediately from Proposition \ref{3.4}.
\end{proof}

\begin{lemma}\label{product}
A direct product $\prod_{i\in I} R_i$ of rings is $UJ^{\#}$ if and only if each direct factor $R_i$ is $UJ^{\#}$.
\end{lemma}

\begin{proof}
As we are aware that \(J^{\#}(\prod_{i \in I} R_i) = \prod_{i \in I} J^{\#}(R_i)\) and \(U(\prod_{i \in I} R_i) = \prod_{i \in I} U(R_i)\), the result follows at once.	
\end{proof}

\begin{lemma}\label{1.5}
Let $R$ be a ring and $e \in Id(R)$. Then, $J^{\#}(eRe)=eRe\cap J^{\#}(R)=eJ^{\#}(R)e$.
\end{lemma}

\begin{proof}
According to the equality $J(eRe)=eRe\cap J(R)=eJ(R)e$ and the definition of $J^{\#}(R)$, the result holds.
\end{proof}

\begin{lemma}\label{corner}
Let $R$ be a $UJ^{\#}$ ring and $e \in Id(R)$. Then, the corner subring $eRe$ is too a $UJ^{\#}$ ring.
\end{lemma}

\begin{proof}
Let $u \in U(eRe)$ with the inverse $v$. Therefore, $u + (1 - e) \in U(R)$ with the inverse $v + (1 - e)$, so that $u + (1 - e) \in 1 + J^{\#}(R)$ and $u - e \in J^{\#}(R) \cap eRe = J^{\#}(eRe)$. Hence, $u \in e + J^{\#}(eRe) = 1_{eRe} + J^{\#}(eRe)$, and thus $eRe$ is a $UJ^{\#}$ ring, as claimed.
\end{proof}

\begin{theorem}\label{3.5}
Let \(I \subseteq J(R)\) be an ideal of a ring \(R\). Then, \(R\) is $UJ^{\#}$ if and only if the quotient \(R/I\) is so.
\end{theorem}

\begin{proof}
Let \(R\) be a $UJ^{\#}$ ring and choose \(u + I \in U(R/I)\). Then, \(u \in U(R)\) and hence \(u = 1 + j\), where \(j \in J^{\#}(R)\). So, \(u + I= (1+I)+(j+I)\), where \(j + I \in J^{\#}(R)/I = J^{\#}(R/I)\) looking at Lemma \ref{1.2}(5).

Oppositely, let \(R/I\) be a $UJ^{\#}$ ring and choose \(u \in U(R)\). Then, \(u + I \in U(R/I)\) and hence \(u + I = (1 + I) + (j + I)\), where \(j + I \in J^{\#}(R/I)=J^{\#}(R)/I\). Thus, \(u + I = (1 + j) + I\). Therefore, \(u - (1 + j) \in I \subseteq J(R) \subseteq J^{\#}(R)\). Consequently, \(u = 1 + j^\prime\), where \(j^\prime \in J^{\#}(R)\). Finally, \(R\) is a $UJ^{\#}$ ring.
\end{proof}

Our new statements are the following ones.

\begin{lemma}\label{close prod}
Let $R$ be a ring. Then, the following two points are fulfilled:
	
(1) If $a\in J^{\#}(R)$ and $b \in J(R)$, then $a+b \in J^{\#}(R)$.
	
(2) If $a,b \in J^{\#}(R)$ and $b \in Z(R)$, then $a + b \in J^{\#}(R)$.
\end{lemma}

\begin{proof}
(1) Suppose that $a \in J^{\#}(R)$ and $b \in J(R)$. Since $a^n \in J(R)$ for some $n$, we deduce
\[
(a+b)^n = a^n + b^n + \sum_{\text{finite}} (\text{mixed products of } a \text{ and } b),
\]
and each term obviously lies in $J(R)$. Therefore, $a+b \in J^{\#}(R)$.

(2) This follows automatically from Lemma \ref{1.2}(4) combined with part (1).
\end{proof}

\begin{lemma}\label{equ UQ}
$R$ is a $UJ^{\#}$ ring if and only if $U(R) + (U(R)\cap Z(R)) = J^{\#}(R)$.
\end{lemma}

\begin{proof}
Evidently, $J^{\#}(R)$ is always a subset of $U(R) + (U(R)\cap Z(R))$. Now, if $R$ is a $UJ^{\#}$ ring and $u \in U(R)$, $v \in U(R) \cap Z(R)$, then $1+u$, $1-v \in J^{\#}(R)$. Therefore, Lemma \ref{close prod} employs to write that $u+v=(1+u)-(1-v) \in J^{\#}(R)$. The reverse implication is also evident, because, for every $u \in U(R)$, we have $u-1 \in J^{\#}(R)$, as requested.
\end{proof}

\begin{proposition}\label{2.2}
Let \( R \) be a $UJ^{\#}$ ring, and set \( \bar{R} := R/J(R) \). The following two points are valid:

(1) For any \( u_1, u_2 \in U(R) \), it follows that \( u_1 + u_2 \neq 1 \).

(2) For any \( \bar{u_1}, \bar{u_2} \in U(\bar{R}) \), it follows that \( \bar{u_1} + \bar{u_2} \neq \bar{1} \).
\end{proposition}

\begin{proof}
(1) Combining Lemma \ref{equ UQ} and Proposition \ref{3.8}, the proof is straightforward.
	
(2) Let us assume $\bar{u}_1 + \bar{u}_2 \neq \bar{1}$. Thus, $u_1, u_2 \in U(R)$ and $u_1 + u_2 - 1 \in J(R)$. Therefore, $u_1 - 1 \in U(R) + J(R)$. However, since $U(R) + J(R) \subseteq U(R)$, there exists $u'_2 \in U(R)$ such that $u_1 + u'_2 = 1$, which contradicts (1), as pursued.
\end{proof}

\begin{proposition}\label{2.3}
Let \( R \) be a potent $UJ^{\#}$ ring, and put \( \bar{R} := R/J(R) \). The following two items are true:

(1) For any \( \bar{e} = \bar{e}^2 \in \bar{R} \) and any \( \bar{u_1}, \bar{u_2} \in U(\bar{e} \bar{R} \bar{e}) \), it must be that \( \bar{u_1} + \bar{u_2} \neq \bar{e} \).

(2) There does not exist \( \bar{e} = \bar{e}^2 \in \bar{R} \) such that \( \bar{e} \bar{R} \bar{e} \cong M_2(S) \) for some ring \( S \).
\end{proposition}

\begin{proof}
(1) Given \( \bar{e}, \bar{u_1}, \bar{u_2} \) as in (1), we can assume that \( e^2 = e \in R \), because idempotents lift modulo \( J(R) \). Then, \( \bar{e} \bar{R} \bar{e} \cong eRe/J(eRe) \). Finally, since \( e R e \) is $UJ^{\#}$, item (1) follows directly from Proposition \ref{2.2}(2).

(2) Note that, in any \( 2 \times 2 \) matrix ring, it is always fulfilled that
		\[
		\begin{pmatrix}
			1 & 0 \\
			0 & 1
		\end{pmatrix}
		= \begin{pmatrix}
			1 & 1 \\
			1 & 0
		\end{pmatrix} + \begin{pmatrix}
			0 & -1 \\
			-1 & 0
		\end{pmatrix} \in U(M_2(S)) + U(M_2(S)).
		\]
Hence, there exist \( \bar{u_1}, \bar{u_2} \in U(\bar{e} \bar{R} \bar{e}) \) such that \( \bar{u_1} + \bar{u_2} = \bar{e} \). This, however, is a contradiction with (1), as asked.
\end{proof}

\begin{lemma}\label{matrix}
For any ring $S \neq 0$ and any integer $n \ge 2$, the matrix ring $M_n(S)$ of size $n$ is {\it not} a $UJ^{\#}$ ring.
\end{lemma}

\begin{proof}
Since $M_2(S)$ is isomorphic to a corner ring of $M_n(S)$ for any $n \ge 2$, it suffices to illustrate that $M_2(S)$ is not a $UJ^{\#}$ ring. To that aim, consider the matrix
    $U=\begin{pmatrix}
        0 & 1\\1 & 1
    \end{pmatrix},$ which is clearly a unit. But, since
    $I-U=\begin{pmatrix}
        1 & -1\\-1 & 0
    \end{pmatrix}$
is a unit too, it cannot be in $J^{\#}(R)$, so that $M_2(S)$ is manifestly not a $UJ^{\#}$ ring, as stated.
\end{proof}

A set $\{e_{ij} : 1 \le i, j \le n\}$ of nonzero elements of $R$ is said to be a system of $n^2$ matrix units provided that $e_{ij}e_{st} = \delta_{js}e_{it}$, where $\delta_{jj} = 1$ and $\delta_{js} = 0$ for $j \neq s$. In this case, $e := \sum_{i=1}^{n} e_{ii}$ is an idempotent of $R$ and $eRe \cong M_n(S)$, where $$S = \{r \in eRe : re_{ij} = e_{ij}r,~~\textrm{for all}~~ i, j = 1, 2, . . . , n\}.$$

Recall that a ring $R$ is said to be {\it Dedekind-finite} if $ab=1$ insures $ba=1$ for any $a,b\in R$. In other words, all one-sided inverses in the ring are two-sided.

\medskip

The next affirmation also appeared in \cite{SU}.

\begin{lemma}\label{dedkind finite}
Every $UJ^{\#}$ ring is Dedekind-finite.
\end{lemma}

\begin{proof}
If we assume the contrary that $R$ is not a Dedekind-finite ring, then there exist elements $a, b \in R$ such that $ab = 1$ but $ba \neq 1$. Assuming $e_{ij} = a^i(1-ba)b^j$ and $e =\sum_{i=1}^{n}e_{ii}$, there exists a nonzero ring $S$ such that $eRe \cong M_n(S)$. However, referring to Lemma \ref{corner}, $eRe$ is a $UJ^{\#}$ ring, so that $M_n(S)$ must too be a $UJ^{\#}$ ring, which contradicts Lemma \ref{matrix}, as it should be.
\end{proof}

We, thereby, extract the following consequence.

\begin{corollary}
Let $R$ be a $UJ^{\#}$ ring and choose $a, b \in R$. Then, $1-ab \in J^{\#}(R)$ if and only if $1-ba \in J^{\#}(R)$.
\end{corollary}

\begin{proof}
Assuming that $1-ab \in J^{\#}(R)$, we have $ab \in U(R)$. Therefore, invoking Lemma \ref{dedkind finite}, $a \in U(R)$. Since $1-ab=(1-ab)aa^{-1} \in J^{\#}(R)$, it follows that $a^{-1}(1-ab)a \in J^{\#}(R)$. Hence, $1-ba=a^{-1}(1-ab)a \in J^{\#}(R)$, as needed.
\end{proof}

Point (1) of next claim is already known in \cite{SU}.

\begin{lemma}\label{2 in J(R)}
Let $R$ be a $UJ^{\#}$ ring. Then, the following two conditions hold:

(1) $2 \in J^{\#}(R)$ and, in particular, $2 \in J(R)$.

(2) $J^{\#}(R)$ is closed under addition uniquely whenever $J^{\#}(R)$ is a subring of $R$.

\end{lemma}

\begin{proof}
(1) It is straightforward, so we omit the details.

(2) Assuming that $J^{\#}(R)$ is closed under addition, it is sufficient to demonstrate that $J^{\#}(R)$ is closed under multiplication. To that goal, let us assume $x, y \in J^{\#}(R)$. Then, we may write
$(1+x)(1+y) \in U(R) = 1+J^{\#}(R)$, which gives $x+y+xy \in J^{\#}(R)$. Moreover, from the assumption $x+y\in J^{\#}(R)$, it follows that $xy \in J^{\#}(R)$. Consequently, $J^{\#}(R)$ is closed under multiplication, as required.
\end{proof}

A concrete consequence is the following one.

\begin{corollary}
The ring $\mathbb{Z}_n$ is $UJ^{\#}$ if and only if $n$ is a power of 2.
\end{corollary}

It is worthwhile noticing that statement (3) of our next claim is published in \cite{SU} as well.

\begin{lemma} The following are fulfilled:
(1) A division ring $R$ is $UJ^{\#}$ if and only if $R \cong \mathbb{F}_2$.

(2) A local ring $R$ is $UJ^{\#}$ if and only if $R/J(R) \cong \mathbb{F}_2$.

(3) A semi-simple ring $R$ is $UJ^{\#}$ if and only if $R \cong \mathbb{F}_2 \times \cdots \times \mathbb{F}_2$.

(4) A semi-local ring $R$ is $UJ^{\#}$ if and only if $R/J(R) \cong \mathbb{F}_2 \times \cdots \times \mathbb{F}_2$.
\end{lemma}

\begin{proof}
(1) If $R$ is a division ring, then we know that $J^{\#}(R)=0$. Therefore, $U(R)={1}$. The converse is obvious.

(2) It is clear by (1).

(3) If $R$ is a semi-simple ring, in conjunction with the well-known Wedderburn-Artin theorem (cf. \cite{lam}), we write $R\cong \prod_i M_{n_i}(D_i)$ for some division rings $D_i$ and finite indices $i$. Consequently, based on Lemmas \ref{product} and \ref{matrix} combined with (1), we conclude that $R \cong \mathbb{F}_2 \times \cdots \times \mathbb{F}_2$, as formulated.

(4) It is obvious by (3).
\end{proof}

A quick consequence is the following.

\begin{corollary}
A local ring $R$ is $UJ^{\#}$ if and only if $R$ is uniquely clean.
\end{corollary}

Our first major criterion is this one, which sounds a bit curiously.

\begin{theorem}\label{m}
Let $R$ be a ring. Then, $R$ is a $UJ^{\#}$ ring if, and only if, $R/J(R)$ is a $UU$ ring.
\end{theorem}

\begin{proof}
Suppose that $R$ is $UJ^{\#}$ and choose $u+J(R)\in U(R/J(R))$. So, $u\in U(R)$ and hence we may write $u=1+j$, where $j\in J^{\#}(R)$. Thus, we can write $u+J(R)=(1+j)+J(R)=(1+J(R)) + (j+J(R))$, where $j+J(R)\in Nil(R/J(R))$, because $j\in J^{\#}(R)$ enables us that $j^n\in J(R)$ for some $n$, as required.

Conversely, suppose that $R/J(R)$ is $UU$ and choose $u\in U(R)$. Hence, $u+J(R)\in U(R/J(R))$. Thus, we may write $u+J(R)=(1+J(R))+(q+J(R))$, where $q+J(R)\in Nil(R/J(R))$, and hence $(q+J(R))^n=J(R)$ for some $n$. Therefore, $q^n\in J(R)$ whence $q\in J^{\#}(R)$. On the other hand, $u-(1+q)\in J(R)$ and so $u=1+q+j$, where $j\in J(R)$. But, $q+j\in J^{\#}(R)$ in virtue of Lemma \ref{close prod}. Consequently, $R$ is $UJ^{\#}$, as required.
\end{proof}

Our two direct consequences are these:

\begin{corollary}
Let $R$ be a ring with $J(R)=(0)$. Then, $R$ is a $UJ^{\#}$ ring if and only if $R$ is a $UU$ ring.
\end{corollary}

A slight expansion of the preceding statement is this one.

\begin{corollary}\label{nilcor}
Let $R$ be a ring with $J(R)$ nil. Then, $R$ is a $UJ^{\#}$ ring if and only if $R$ is a $UU$ ring.
\end{corollary}

We close this section with the following two constructions.

\begin{example}
(1) Every $UJ$ ring is $UJ^{\#}$, but the converse is {\it not} true in general. Indeed, let $R$ be the $\mathbb{F}_2$-algebra generated by $x, y$ and with $x^2=0$. Consulting with \cite{DL}, we have that $J(R)=(0)$, $U(R)=1+\mathbb{F}_2x+xRx$ and $N(R)=\mathbb{F}_2x+xRx$. Then, $R$ is UU and hence $UJ^{\#}$ but definitely is {\it not} UJ.

(2) Every UU ring is $UJ^{\#}$, but the converse is {\it not} true in general. In fact, $R=\mathbb{F}_2[[x]]$ is $UJ^{\#}$ but surely is {\it not} UU in view of Corollary~\ref{nilcor}, because $J(R)$ is not nil.
\end{example}

\section{Main Results}

Standardly, a ring $R$ is called {\it semi-potent} if every one-sided ideal not contained in $J(R)$ contains a non-zero idempotent. Moreover, a semi-potent ring $R$ is said to be {\it potent} if idempotents lift modulo $J(R)$. Besides, as usual, a ring $R$ is called {\it reduced} if it contains no non-zero nilpotent elements.

\medskip

We are now in a position to establish our second chief result.

\begin{theorem}\label{2.4}
Let \( R \) be a semi-potent ring. Then, the following three statements are equivalent:

(1) \( R \) is a $UJ^{\#}$ ring.

(2) \( R/J(R) \) is Boolean.

(3) \( R \) is a \( UJ \) ring.
\end{theorem}

\begin{proof}
\( (1) \Rightarrow (2) \). Since $R$ is a semi-potent ring, we know that the quotient $\overline{R} = R/J(R)$ is also a semi-potent ring. We will prove now that $\overline{R}$ is a reduced ring.

To that end, assume $x^2 = 0$ but $0 \neq x \in \overline{R}$. So, employing \cite[Theorem 2.1]{Levi}, there exists $e \in \overline{R}$ such that $e\overline{R}e \cong M_2(S)$, where $S$ is a non-zero ring. However, since by assumption $\overline{R}$ is an $UJ^{\#}$ ring, Lemma \ref{corner} employs to deduce that $e\overline{R}e$ is an $UJ^{\#}$ ring. But, on the other hand, Lemma \ref{matrix} tells us that $M_2(S)$ is not a $UJ^{\#}$ ring, which is a contradiction. Therefore, $\overline{R}$ is indeed a reduced ring.

Now, assume there exists $x \in \overline{R}$ such that $x - x^2 \neq 0$ in $\overline{R}$. Since $\overline{R}$ is a semi-potent ring, there exists $e = e^2 \in \overline{R}$ such that $e \in (x - x^2)\overline{R}$. Thus, $e = (x - x^2)y$ for some $y \in \overline{R}$. Since $e$ is central (indeed, as $[er(1-e)]^2 = 0 = [(1-e)re]^2$, we thus derive $er(1-e) = 0 = (1-e)re$), we can write $e = ex \cdot e(1-x) \cdot ey$, so that both $ex, e(1-x) \in U(e\overline{R}e)$. But, Lemma \ref{corner} informs us that $e\overline{R}e$ is an $UJ^{\#}$ ring. However, $ex + e(1-x) = e$, which obviously contradicts Proposition \ref{2.2}. Consequently, $\overline{R}$ is a boolean ring, as asserted.
	
\( (2) \Rightarrow (3) \). Let us assume \( u \in U(R) \). Then, \( \bar{u} \in U(\bar{R} = R/J(R)) \). Since \( \bar{R} \) is a Boolean ring, we have \( \bar{u} = \bar{1} \), guaranteeing that \( u-1 \in J(R) \), as requested.
	
\( (3) \Rightarrow (1) \). This is rather obvious, so we drop off the argumentation.
\end{proof}

Two more consequences are these:

\begin{corollary}\label{2.5}
A regular ring \( R \) is $UJ^{\#}$ if and only if \( R \) is \( UJ \) if and only if \( R \) is \( UU \) if and only if \( R \) is Boolean.
\end{corollary}

\begin{proof}
Since $R$ is regular, one has that $J(R)=(0)$ and $R$ is semi-potent. So, the outcome follows from Theorem \ref{2.4}.	 \end{proof}

\begin{corollary}\label{2.7}
Let \( R \) be an Artinian (or, in particular, a finite) ring. Then, the following three issues are equivalent:

(1) \( R \) is a $UJ^{\#}$ ring.

(2) \( R \) is a \( UJ \) ring.

(3) \( R \) is a \( UU \) ring.
\end{corollary}

\begin{proof}
We know that every Artinian ring is always clean. Moreover, since \( R \) is Artinian, we have \( J(R) \subseteq Nil(R) \), as we need. The partial case is immediate, because any finite ring is known to be Artinian.
\end{proof}

Our third main result is the following.

\begin{theorem}\label{3.16}
Let \(R\) be a ring. Then, the following three claims are equivalent:

(1) \(R\) is a semi-regular $UJ^{\#}$ ring.

(2) \(R\) is an exchange $UJ^{\#}$ ring.

(3) \(R\) is a semi-Boolean ring.
\end{theorem}

\begin{proof}
\((1) \Rightarrow (2)\). It is very simple, because semi-regular rings are always exchange.
	
\((2) \Rightarrow (3)\). This is routine utilizing Theorem \ref{2.4}.
	
\((3) \Rightarrow (1)\). Since $R$ is semi-Boolean, by definition, the factor-ring $R/J(R)$ is Boolean and idempotents lift modulo $J(R)$. So, $R$ is semi-regular. Furthermore, Theorem \ref{2.4} is a guarantor that $R/J(R)$ is $UJ^{\#}$ whence $R$ is too $UJ^{\#}$ invoking Theorem \ref{3.5}.
\end{proof}

Our two next consequences are derived as follows.

\begin{corollary}\label{3.17}
Let \( R \) be a $UJ^{\#}$ ring. Then, the following three assertions are equivalent:

(1) \( R \) is a semi-regular ring.

(2) \( R \) is an exchange ring.

(3) \( R \) is a clean ring.
\end{corollary}

\begin{proof}
Applying Theorem \ref{3.16}, (1) \(\Leftrightarrow\) (2) holds.
	
\((3) \Rightarrow (2)\). This is quite apparent, so we omit the arguments.
	
\((2) \Rightarrow (3)\). If \( R \) is exchange $UJ^{\#}$, then it follows that \( R \) is reduced, and hence it is abelian. Therefore, \( R \) is abelian exchange, so it is known to be clean.
\end{proof}

Now, we obtain a new characterization for strongly nil-clean rings.

\begin{corollary}\label{3.18}
Let \( R \) be a ring. Then, the following three statements are equivalent:

(1) \( R \) is a semi-regular $UJ^{\#}$ ring and \(J(R)\) is nil.

(2) \( R \) is an exchange $UJ^{\#}$ ring and \(J(R)\) is nil.

(3) \( R \) is a strongly nil-clean ring.
\end{corollary}

\begin{proof}
\((1) \Rightarrow (2)\). This is rather elementary, so we leave out the details.
	
\((2) \Rightarrow (3)\). We know that $Nil(R)\subseteq J^{\#}(R)$. Now, let $a\in J^{\#}(R)$. Thus, $a^n \in J(R)$ for some $n$. By hypothesis, $a^n\in Nil(R)$ and hence $a\in Nil(R)$. So, we have $J^{\#}(R) \subseteq Nil(R)$. Therefore, we have, $Nil(R)=J^{\#}(R)$. We, thereby, conclude that $R$ is a UU ring and, likewise, it is exchange. Hence, $R$ is strongly nil-clean mimicking \cite[Theorem 3.8]{DL}.
	
\((3) \Rightarrow (1)\). Knowing that $J(R)$ is nil, we thus infer that all idempotents lift modulo $J(R)$. Also, $R/J(R)$ is a Boolean ring looking at \cite[Theorem 3.8]{DL}. Hence, $R$ is a semi-regular $UJ^{\#}$ ring following Theorem \ref{3.16}.
\end{proof}

We now consider some element-wise properties in $UJ^{\#}$ rings.

\begin{proposition}\label{pro clean element}
For any $UJ^{\#}$ ring $R$, the following three affirmations are valid:
	
(1) An element $a$ in $R$ is clean if and only if it is $J^{\#}$-clean.
	
(2) An element $a$ in $R$ is strongly clean if and only if it is strongly $J^{\#}$-clean.
\end{proposition}

\begin{proof}
We will only prove part (1), as the other case is rather analogous. To that target, let us assume that $a \in R$ is clean with a clean decomposition $a = e + u$. Since $R$ is a $UJ^{\#}$ ring, we have $u = 1 + j$, where $j \in J^{\#}(R)$. Therefore, $a = (1 - e) + (2e + j)$. So, since $$(1 + j) + 2e \in U(R) + J(R) \subseteq U(R) = 1 + J^{\#}(R),$$ we conclude that $2e + j \in J^{\#}(R)$, as required.
	
Conversely, if $a \in R$ is $J^{\#}$-clean and $a = e + j$ is a $J^{\#}$-clean decomposition, we have $a = (1 - e) - (1 - (2e + j))$. Similarly to the previous case, we can easily show that $1 - (2e + j) \in U(R)$, as requested.
\end{proof}

Three other consequences are the following.

\begin{corollary}\label{impor cor equ}
Suppose that $R$ is an arbitrary ring. Then, the following three conditions are equivalent:
	
(1) $R$ is a $UJ^{\#}$ ring.
	
(2) Every clean element is a strongly $J^{\#}$-clean.
	
(3) For every $u \in U(R)$, there is $e = e^2 \in Z(R)$ and $j \in J^{\#}(R)$ such that $u = e + j$.
\end{corollary}

\begin{proof}
(1) $\Rightarrow$ (2). It follows at once from Proposition \ref{pro clean element}.
	
(2) $\Rightarrow$ (1). Given $u \in U(R)$, since $u$ is strongly clean, taking into account point (2) we write $u = e + j$, where $e \in Id(R)$, $j \in J^{\#}(R)$ and $ej = je$. Thus, $u^{-1}e = 1 + u^{-1}j \in U(R)$. This shows that $e \in U(R) \cap Id(R) = \{1\}$, as wanted.
	
(1) $\Rightarrow$ (3). It just suffices to substitute $e = 1$.
	
(3) $\Rightarrow$ (1). Assume $u \in U(R)$. So, $u = e + j$ and, therefore, $u^{-1}e = 1 + u^{-1}j \in U(R)$. This shows that $e = 1$, as desired.
\end{proof}

\begin{corollary}\label{cor1}
Each strongly $J^{\#}$-clean ring is $UJ^{\#}$.
\end{corollary}

\begin{proof}
The conclusion can be obtained directly from Corollary \ref{impor cor equ}.
\end{proof}

\begin{corollary}\label{1.6}
Suppose that $R$ is an arbitrary ring. Then, the following three conditions are equivalent:
	
(1) $R$ is a clean $UJ^{\#}$ ring.
	
(2) $R$ is a $J^{\#}$-clean $UJ^{\#}$ ring.
	
(3) $R$ is a $J^{\#}$-clean ring.
\end{corollary}

The last statement unambiguously shows that statement (2) is superfluous, or more exactly that in it the condition of being a $UJ^{\#}$ ring is {\it not} necessary.

\medskip

The next claim somewhat supplies Corollary~\ref{cor1}.

\begin{lemma}\label{cor2}
Any strongly $J^{\#}$-clean ring is strongly clean.
\end{lemma}

\begin{proof}
Let $a \in R$ be strongly $J^{\#}$-clean. Then, there are $e^2 = e \in R$ and $j \in J^{\#}(R)$ such that $a = e + j$ and $ej = je$. Thus, we write $a = (1 - e) + (2e - 1 + j)$. With Lemma \ref{2 in J(R)} in mind, we have $2 \in J(R)$, so $2e -1 + j \in J(R) + U(R) \subseteq U(R)$, as needed.
\end{proof}

We now can discover the following.

\begin{corollary}
A ring $R$ is strongly $J^{\#}$-clean if and only if the following two points are fulfilled:
	
(1) $R$ is $UJ^{\#}$.
	
(2) $R$ is strongly clean.
\end{corollary}

\begin{proof}
If $R$ is a strongly $J^{\#}$-clean ring, the claim is concluded from Corollary \ref{cor1} and Lemma \ref{cor2}. As for the converse, if $R$ is both $UJ^{\#}$ and strongly clean, Proposition \ref{pro clean element} is applicable to get that $R$ is a strongly $J^{\#}$-clean ring, as promised.
\end{proof}

Let $Nil_*(R)$ denote the {\it prime radical} (or, in other terms, the {\it lower nil-radical}) of a ring $R$, i.e., the intersection of all prime ideals of $R$. We know that $Nil_*(R)$ is a nil-ideal of $R$. Moreover, a ring $R$ is called {\it $2$-primal} if $Nil_*(R)$ consists precisely of all the nilpotent elements of $R$ (i.e., provided that $Nil(R)=Nil_*(R)$). For example, it is well known that all reduced rings and commutative rings are $2$-primal.

\medskip

Let \( R \) be a ring and let \( \alpha \) be an endomorphism of \( R \). Recall that \( R \) is called \textit{\(\alpha\)-compatible}, provided that, for all \( a, b \in R \), we have $ab = 0$ if, and only if, $a\alpha(b) = 0$.

\medskip

We can now attack the following statement.

\begin{proposition}\label{2.10}
Let \( R \) be a \(2\)-primal ring and \( \alpha \) an endomorphism of \( R \). If \( R \) is \( \alpha \)-compatible, then
\[
	J^{\#}(R[x; \alpha]) = Nil_*(R[x; \alpha]).
\]
\end{proposition}

\begin{proof}
It is immediate that $Nil_*(R[x; \alpha])\subseteq J^{\#}(R[x; \alpha])$. Hence, it suffices to show that $J^{\#}(R[x; \alpha])\subseteq Nil_*(R[x; \alpha])$. To that purpose, assume that \( f \in J^{\#}( R[x; \alpha] ) \). Then, there exists \( n \in \mathbb{N} \) such that \( f^n \in J( R[x; \alpha] ) \). Hence, \[ f^n \in J(R[x; \alpha]) = Nil_*(R)[x] = Nil_*(R[x; \alpha]) \subseteq Nil(R[x; \alpha]).\] Therefore, \( f \in Nil(R[x; \alpha]) \).

On the other hand, we have $Nil(R[x; \alpha]) \subseteq Nil(R)[x; \alpha]$. But, since \(R\) is a 2-primal ring, it must be that \(f \in Nil(R[x; \alpha]) \subseteq Nil(R)[x; \alpha]=Nil_*(R)[x; \alpha]\), as asked for.
\end{proof}

Two more consequences are as follows.

\begin{corollary}\label{4.1}
Let \( R \) be a \(2\)-primal ring and \( \alpha \) an endomorphism of \( R \). If \( R \) is \( \alpha \)-compatible, then
\[
	J^{\#}(R[x; \alpha])=Nil_*(R[x; \alpha])=Nil_*(R)[x; \alpha]=Nil(R)[x; \alpha]=J(R[x; \alpha]).
\]
\end{corollary}

\begin{corollary}\label{2.40}
Let \( R \) be a \(2\)-primal ring. Then
\[
	J^{\#}(R[x])=Nil_*(R[x])=Nil_*(R)[x]=Nil(R)[x]=J(R[x]).
\]
\end{corollary}

Our basic result in this light is this one.

\begin{theorem}\label{skew polynomial}
Let \( R \) be a \(2\)-primal ring and let \( \alpha \) be an endomorphism of \( R \). If \( R \) is \( \alpha \)-compatible, then the following four assertions are equivalent:

(1) $R$ is a UU-ring.
	
(2) $R[x; \alpha]$ is a $UJ^{\#}$-ring.
			
(3) $R[x; \alpha]$ is a UJ-ring.
	
(4) $R[x; \alpha]$ is a UU-ring.
\end{theorem}

\begin{proof}
(1) $\Leftrightarrow$ (4). This is immediate directing to \cite[Theorem 2.11]{karimi}.
	
(2) $\Leftrightarrow$ (3). It follows at once from Corollary \ref{4.1}.
	
(1) $\Rightarrow$ (2). Let us assume $f=\sum_{i=0}^{n}a_ix^i \in U(R[x; \alpha])$. Since $R$ is a $2$-primal ring, \cite[Corollary 2.14]{Chenpr} leads to $a_0 \in U(R)$ and $a_i \in Nil(R)$ for all $1 \le i \le n$. However, since $R$ is a UU ring, we find that $$1 - f = (1 - a_0) + \sum_{i=1}^{n}a_ix^i \in Nil(R)[x; \alpha] = J^{\#}(R[x; \alpha]).$$
	
(2) $\Rightarrow$ (1). Let us assume $u \in U(R)$. Thus, we can write $$u \in U(R[x; \alpha]) = 1 + J^{\#}(R[x; \alpha]) = 1 + Nil(R)[x; \alpha].$$ Therefore, $u - 1 \in Nil(R)$.
\end{proof}

Our next pivotal result is the following one.

\begin{theorem}\label{thm 2 primal}
Let $R$ be a ring. Then, the following two items are true:
	
(1) $R$ is a $2$-primal ring if and only if $J^{\#}(R[x]) = Nil_*(R)[x]$.
	
(2) $R$ is a reduced ring if and only if $J^{\#}(R[x]) = \{0\}$.
\end{theorem}

\begin{proof}
(1) Letting $R$ be a $2$-primal ring, the result follows in accordance with Lemma \ref {2.40}.
	
Conversely, let us assume that $J^{\#}(R[x]) = Nil_*(R)[x]$. If $a \in Nil(R)$ and $a^n = 0$, then $a^n \in J^{\#}(R[x])$. The application of Lemma \ref{1.2}(2) means that $a \in J^{\#}(R[x]) = Nil_*(R)[x]$. Thus, $a \in Nil_*(R)$.
	
(2) Based on what we have proved in the first part above, nothing remains to be showed.
\end{proof}

We thus perceive the following necessary and sufficient condition.

\begin{proposition}\label{3.2}
A ring $R$ is $UJ^{\#}$ if and only if so is $R[[x; \alpha]]$.
\end{proposition}

\begin{proof}
Consider $I:= R[[x; \alpha]]x$. Then, one inspects that $I$ is an ideal of $R[[x; \alpha]]$. Note also that $J(R[[x; \alpha]])=J(R)+I$, so $I\subseteq J(R[[x; \alpha]])$. Since $R[[x; \alpha]]/I\cong R$, the result follows at once from Theorem \ref{3.5}, as we suspected.
\end{proof}

Let $R$ be a ring and let $M$ be a bi-module over $R$. The {\it trivial extension} of $R$ and $M$ is introduced as
\[ T(R, M) = \{(r, m) : r \in R \text{ and } m \in M\}, \]
with addition defined component-wise and multiplication defined by
\[ (r, m)(s, n) = (rs, rn + ms). \]
Note that the trivial extension $T(R, M)$ is always isomorphic to the (proper) subring
\[ \left\{ \begin{pmatrix} r & m \\ 0 & r \end{pmatrix} : r \in R \text{ and } m \in M \right\} \]
of the formal $2 \times 2$ matrix ring $\begin{pmatrix} R & M \\ 0 & R \end{pmatrix}$, and likewise $T(R, R) \cong R[x]/\left\langle x^2 \right\rangle$. We, also, notice that the set of units of the trivial extension $T(R, M)$ is
\[ U(T(R, M)) = T(U(R), M). \]

Furthermore, a {\it Morita context} is introduced as a 4-tuple $\begin{pmatrix} A & M \\ N & B \end{pmatrix}$, where $A$ and $B$ are rings, $_AM_B$ and $_BN_A$ are bi-modules, and moreover there exist context products $M\times N \to A$ and $N\times M \to B$ written multiplicatively as $(w, z) = wz$ and $(z, w) = zw$. Thus, the Morita context $\begin{pmatrix} A & M \\ N & B \end{pmatrix}$ forms an associative ring with the usual matrix operations.

A Morita context is referred to as {\it trivial} if the two context products are both trivial, meaning that $MN = (0)$ and $NM = (0)$ (see cf. \cite[p. 1993]{Mari}). In this case, we have the isomorphism
$$\begin{pmatrix} A & M \\ N & B \end{pmatrix} \cong T(A \times B, M\oplus N),$$
where $\begin{pmatrix} A & M \\ N & B \end{pmatrix}$ represents a trivial Morita context as stated in \cite{20}.

\medskip

We now are ready to state without evidence our final assertion for this section.

\begin{proposition}\label{trivial}
Suppose $R$ is a ring and $M$ is a bi-module over $R$. Then, the following three items hold:

(1) The trivial extension $T(R, M)$ is a $UJ^{\#}$ ring if and only if $R$ is a $UJ^{\#}$ ring.

(2) The formal triangular matrix ring $\begin{pmatrix} R & N \\ 0 & S \end{pmatrix}$ is a $UJ^{\#}$ ring if and only if $R$ and $S$ are both $UJ^{\#}$ rings.

(3) For all $n \in \mathbb{N}$, the triangular matrix ring $T_{n}(R)$ is a $UJ^{\#}$ ring if and only if $R$ is a $UJ^{\#}$ ring.
\end{proposition}

\section{Group $UJ^{\#}$ Rings}

Imitating the traditional terminology, we say that a group is a {\it $p$-group} if each its element has finite order which is a power of the prime number $p$. Moreover, a group $G$ is said to be {\it locally finite} if every finitely generated subgroup is finite.

Suppose now that $G$ is an arbitrary group and $R$ is an arbitrary ring. As usual, $RG$ stands for the group ring of $G$ over $R$. The homomorphism $\varepsilon :RG\rightarrow R$, defined by the equality $\varepsilon (\displaystyle\sum_{g\in G}a_{g}g)=\displaystyle\sum_{g\in G}a_{g}$, is called the {\it augmentation map} of $RG$ and its kernel, designed by $\Delta (RG)$, is called the {\it augmentation ideal} of $RG$.

\medskip

Before stating and proving our key result, we need of a series of preliminary ingredients.

\begin{lemma}\label{sequ}
Let $R$ be a $UJ^{\#}$ ring, $a \in U(R)$ and $g_n = 1 + a + a^2 + \cdots + a^n$. Then, the following two points occur:
	
(1) If $n$ is even, then $g_n \in U(R)$.
	
(2) If $n$ is odd, then $g_n \in J^{\#}(R)$.
\end{lemma}

\begin{proof}
(1) Suppose $n$ is even. We, then, can write $g_n$ as $$g_n = 1+ (1 + a)(a + a^3 + \cdots + a^{n-1}).$$ Since $a \in U(R)$, we know $a+1 \in J^{\#}(R)$, so that Lemma \ref{1.2}(1) can be applied to infer that $$(1 + a)(a + a^3 + \cdots + a^{n-1}) \in J^{\#}(R).$$ Hence, $g_n \in U(R)$.
	
(2) Suppose $n$ is odd. We can write $g_n$ as $g_n = 1+ ag_{n-1}$. So, the usage of (1) guarantees that $g_{n-1} \in U(R)$ and, therefore, $$g_n = 1 + a g_{n-1} \in 1 + U(R) = J^{\#}(R).$$
\end{proof}

\begin{lemma}\cite[Proposition 9]{con}\label{ext le}
Let $R$ be a ring, $G$ a group and $H$ a subgroup of $G$. Then, the following two items occur:
	
(1) $J(RG) \cap RH \subseteq J(RH)$.
	
(2) If $G$ is a locally finite group, then, $J(R) = J(RG) \cap R$. In particular, $J(R)G \subseteq J(RG)$.
\end{lemma}

\begin{lemma}\label{0}
Let $RG$ be a $UJ^{\#}$ ring. Then, $G$ is necessarily a torsion group.
\end{lemma}

\begin{proof} We present two different ideas for evidence:

\medskip

\textbf{Method 1:} Suppose there is $g \in G$ of infinite order. Since $RG$ is a $UJ^{\#}$ ring, we have $1-g \in J^{\#}(RG)$. Consequently, there is $n \in \mathbb{N}$ such that $$(1-g)^n \in J(RG) \cap R\langle g\rangle \subseteq J(R\langle g\rangle).$$ Furthermore, as $1-g$ is central in the ring $R\langle g\rangle$, we obtain $1-g \in J(R\langle g\rangle)$. So, $1-g+g^2 \in U(R\langle g\rangle)$, and thus there are integers $n < m$ and elements $a_i$ with $a_n \neq 0 \neq a_m$ such that
\[
	(1 - g + g^2)\sum_{i=n}^{m}a_ig^i=1.
\]
This, however, leads to a contradiction, and so every element $g \in G$ must have finite order, as suspected.

\medskip

\textbf{Method 2:} Let us now assume $g \in G$. Then, a simple check shows that $R\left\langle g \right\rangle$ is a rationally closed subring of $RG$. Thus, $R\left\langle g \right\rangle$ is a $UJ^{\#}$ ring. If, however, we assume in a way of contradiction that $\left\langle g \right\rangle$ is an infinite cyclic group, then the application of Lemma \ref{sequ}(1) teaches us that $1 + g + g^2 \in U(RG)$ and hence, using a similar argument to that exploited above, one can readily derive the expected contradiction.
\end{proof}

We are now able to establish the following chief statement.

\begin{theorem}\label{UQgroupring}
Let $R$ be a ring and let $G$ be a group. If $RG$ is a $UJ^{\#}$-ring, then $R$ is a $UJ^{\#}$ ring and $G$ is a $2$-group.
\end{theorem}

\begin{proof}
Since $R$ is always a rationally closed subring of $RG$, it follows directly from Lemma \ref{subring} that $R$ is a $UJ^{\#}$ ring. Now, we observe with the help of Lemma~\ref{sequ} that, for any \( g \in G \) and \( k \in \mathbb{N} \), it must be that \( \sum_{i=0}^{2k} g^i \in U(RG) \). 
	
Next, if \( g \in G \) has order \( p \) that does not divide 2, then \( p \) is obviously odd, whence \( p-1=2k \). Consequently, \(\sum_{i=0}^{2k} g^i \in U(RG)\), and since \( (1-g)(\sum_{i=0}^{2k} g^i)=0 \), we find \( 1-g=0 \), which is a contradiction. Thus, \( G \) must be a 2-group, indeed.
\end{proof}

Further, we need some other technical claims before presenting our final key result in this section.

\begin{lemma}\label{Delta subset Jaco}
Let $G$ be a locally finite $2$-group and let $R$ be a $UJ^{\#}$ ring. Then, $\Delta(RG) \subseteq J(RG)$.
\end{lemma}

\begin{proof}
Suppose that $\overline{R}=R/J(R)$. So, thanking Lemma \ref{2 in J(R)}, $\bar{2} \in N(\overline{R})$. Therefore, with \cite[Proposition 16]{con} in hand, we derive $\Delta(\overline{R}G) \subseteq J(\overline{R}G)$.
	
On the other hand, with \cite[Lemma 4]{chennot} in mind, we deduce $J(R)G \subseteq J(RG)$, so one verifies that $$J(\overline{R}G)=J(R/J(R)G) \cong J(RG/J(R)G)=J(RG)/J(R)G.$$
Now, letting $f:=\sum a_g(1-g) \in \Delta(RG)$, it is simply seen that $$\sum \Bar{a_g}(1-g) \in  \Delta(\overline{R}G) \subseteq J(RG)/J(R)G.$$ Hence, there is $j \in J(RG)$ such that $f-j \in J(R)G \subseteq J(RG)$ and, because $J(RG)$ is an ideal, we conclude that $f$ belongs to $J(RG)$, completing the arguments after all.
\end{proof}

The following two criteria sound satisfactory.

\begin{proposition}\label{locallyfinite}
Let $G$ be a locally finite $2$-group. Then, $RG$ is a $UJ^{\#}$ ring if and only if $R$ is a $UJ^{\#}$ ring.
\end{proposition}

\begin{proof}
Assume $RG$ is $UJ^{\#}$. Since $R$ is always a rationally closed subring of $RG$, it follows that $R$ is a $UJ^{\#}$ ring. 

Now, to treat the reverse implication, suppose $R$ is $UJ^{\#}$. With Lemma \ref{Delta subset Jaco} at hand, we have $\Delta(RG) \subseteq J(RG)$ and, on the other side, we have $RG/\Delta(RG) \cong R$. Therefore, the desired result follows at once from Theorem \ref{3.5}.
\end{proof}

\begin{proposition}
Let $R$ be an Artinian ring. Then, $RG$ is a $UJ^{\#}$ ring if and only if $(R/J(R))G$ is a $UJ^{\#}$ ring.
\end{proposition}

\begin{proof}
Since $R$ is an Artinian ring, we may follow \cite[Proposition 3]{con} to extract that $J(R)G\subseteq J(RG)$. On the other hand, we know that the isomorphism $(R/J(R))G\cong RG/J(R)G$ is always true. Thus, the outcome follows immediately from Theorem \ref {3.5}.
\end{proof}

\begin{lemma}
Let $RG$ be a $UJ^{\#}$ ring with $3 \in J^{\#}(R)$, and let $G$ be a $2$-group. Then, $G$ is a group of exponent $2$.
\end{lemma}

\begin{proof}
Owing to Lemma \ref{1.2}(4), we can get that $3 \in J(R)$. Hence, $4\in U(R)\subseteq U(RG)$. Therefore, $4\in 1+J^{\#}(RG)$ and so $3\in J^{\#}(RG)$. Again, according to Lemma \ref{1.2}(4), we have $3\in J(RG)$.
	
First, we prove that, for every $g \in G$ and $k \in \mathbb{N}$, $1+g^{2^k} \in U(RG)$. To that end, since $R\langle g\rangle$ is a subring of $RG$, Lemma \ref{subring} reaches us that $R\langle g\rangle$ is also a $UJ^{\#}$ ring. Therefore, without loss of generality, we may assume $g$ is central.
	
As $G$ is a $2$-group, let $k \in \mathbb{N}$ be the smallest integer such that $g^{2^k}=1$. Consequently, $1-g^{2^k}=0$. Since $3 \in J(RG)$, we have
\[
	1+g^{2^k} = 1-g^{2^k} - g^{2^k} + 3g^{2^k} = -g^{2^k} + 3g^{2^k} \in U(RG) + J(RG) \subseteq J(RG).
\]
	
Now, consider the equality $(1-g^{2^{k-1}})(1+g^{2^{k-1}})=0$. Since $1+g^{2^{k-1}} \in U(RG)$, we can deduce that $1=g^{2^{k-1}}$, which manifestly leads to a contradiction. Therefore, $G$ must be a group of exponent $2$, as wanted.
\end{proof}

We finish off our work with the following assertion.

\begin{theorem}\label{finone3}
Let $R$ be a ring with $3 \in J^{\#}(R)$, and let $G$ be a $p$-group with $p\neq 2$ a prime. If $RG$ is a $UJ^{\#}$ ring, then $G$ is a $3$-group.
\end{theorem}

\begin{proof}
If $p \neq 2$, then $p$ must be odd. Let $g \in G$ with $g^{p^n}=1$. Suppose $p^n=2k+1$. Utilizing a similar argument as in the proof of Lemma~\ref{0}, we can conclude that
\[
	1-g\in J(R\langle g\rangle).
\]
This gives $\Delta(R\langle g\rangle) \subseteq J(R\langle g\rangle)$ by taking into account that $$1-g^n=(1-g)(1+g + \cdots + g^{n-1}) \in J(R\langle g\rangle).$$ Furthermore, \cite[Proposition 15(i)]{con} allows us to derive that $\langle g\rangle$ is a $q$-group with $q \in J(R)$. But, since $3 \in J(R)$ and two distinct primes cannot be in $J(R)$ simultaneously, we infer that $q=3$. Thus, in turn, there is $m \in \mathbb{N}$ such that $g^{3^m}=1$. As the element $g \in G$ was arbitrary, one detects that $G$ is really a $3$-group, as asserted.
\end{proof}

\medskip
\medskip

\noindent{\bf Funding:} The work of the first-named author, P.V. Danchev, is partially supported by the project Junta de Andaluc\'ia under Grant FQM 264.

\end{document}